\documentclass[11pt,a4paper,twoside]{article}
\usepackage[UKenglish]{babel}
\usepackage[T1]{fontenc}
\usepackage[utf8x]{inputenc}
\usepackage{latexsym,amsfonts,amsmath,amsthm,amssymb,mathrsfs}
\usepackage{fullpage}
\usepackage{xcolor,bm, bbm}
\usepackage{paralist}
\usepackage{todonotes}
\usepackage{dsfont}
\usepackage{float}
\usepackage{caption}
\usepackage{xfrac}

\usepackage{graphicx}
\usepackage{tikz,ifthen}
\usepackage{pgfplots}
\pgfplotsset{compat=1.18} 
\usepgfplotslibrary{fillbetween}

\DeclareMathOperator{\R}{\mathbb{R}}
\DeclareMathOperator{\N}{\mathbb{N}}
\DeclareMathOperator{\Z}{\mathbb{Z}}
\DeclareMathOperator{\Q}{\mathbb{Q}}

\newtheorem{thm}{Theorem}%[section]

\newtheorem{lem}{Lemma}[section]
\newtheorem{conj}{Conjecture}
\newtheorem{cor}[thm]{Corollary}

\newtheorem{proposition}[lem]{Proposition}

\newtheorem*{rem*}{Remark}
\newtheorem*{thm*}{Theorem}

\title{\bf The Duffin-Schaeffer Conjecture for multiplicative Diophantine approximation}
\author{Lorenz Fr\"uhwirth and Manuel Hauke}
\medskip

\newenvironment{dedication}
        {\vspace{6ex}\begin{quotation}\begin{center}\begin{em}}
        {\par\end{em}\end{center}\end{quotation}}

\date{}

\begin{document}

%\subjclass{11K60 11J83 11J71 11H16}
%keywords: metric Diophantine approximation, multiplicative Diophantine approximation, Duffin-Schaeffer conjecture, Littlewood's conjecture

\maketitle
\vspace{-15mm}

\begin{dedication}
Dedicated to the memory of Aleksandr Khintchine, on the occasion of the
100th anniversary of his foundational theorem in metric Diophantine
approximation.
\end{dedication}

\begin{abstract}
Given a monotonically decreasing $\psi: \N \to [0,\infty)$, Khintchine's Theorem provides an efficient tool to decide whether, for almost every $\alpha \in \R$, there are infinitely many $(p,q) \in \Z^2$ such that
$\left\lvert \alpha - \frac{p}{q}\right\rvert \leq \frac{\psi(q)}{q}$. The recent result of Koukoulopoulos and Maynard provides an elegant way of removing monotonicity when only counting reduced fractions. Gallagher showed a multiplicative higher-dimensional generalization to Khintchine's Theorem, again assuming monotonicity. 
In this article, we prove the following Duffin-Schaeffer-type result for multiplicative approximations: For any $k\geq 1$, any function $\psi: \N \to [0,1/2]$ (not necessarily monotonic) and almost every $\alpha \in \mathbb{R}^k$, there exist infinitely many $q$ such that $\prod\limits_{i=1}^k \left\lvert \alpha_i - \frac{p_i}{q}\right\rvert \leq \frac{\psi(q)}{q^k}, p_1,\ldots,p_k$ all coprime to $q$, if and only if 
\[\sum\limits_{q \in \N} \psi(q) \left(\frac{\varphi(q)}{q} \right)^k\log \left(\frac{q}{\varphi(q)\psi(q)}\right)^{k-1} = \infty.\]
This settles a conjecture of Beresnevich, Haynes, and Velani.
\end{abstract}

\section{Introduction and main results}

\subsection{The metric theory of Diophantine approximation}\label{classical_DA}
Given a function $\psi: \N \to [0,\infty)$ and a (typically irrational) number $\alpha \in [0,1]$, the main goal of Diophantine approximation is to try to understand whether there are infinitely many $(a,q) \in \Z \times \N$ such that
\begin{equation}
\label{central_question}
\left\lvert \alpha - \frac{a}{q}\right\rvert \leq \frac{\psi(q)}{q}.
\end{equation}
In case of $\psi(q)=\frac{1}{q}$, Dirichlet's Theorem shows that \eqref{central_question} holds for all $\alpha \in [0,1]$. This is known to be asymptotically sharp for badly approximable numbers, but Khintchine's Theorem tells us that we can improve upon Dirichlet's Theorem for Lebesgue almost every $\alpha \in \R$:

\begin{thm*}[Khintchine, 1924]
    Let $\psi: \N \to [0,\infty)$ be a \textbf{monotonically decreasing} function.
   Writing 
   \[S(\psi) := \left\{\alpha \in [0,1]: \lVert q\alpha \rVert \leq \psi(q), \text{ for i.m. }q \in \N\right\},\] we have
    \[
   \lambda_1\left(S(\psi)\right)
   = \begin{cases}
       1, &\text{ if } \sum\limits_{q \in \N} \psi(q) = \infty,\\
       0, &\text{ if } \sum\limits_{q \in \N}\psi(q) < \infty.
   \end{cases}\]
\end{thm*}

Here and in what follows ``i.m.'' stands for ``infinitely many'', $\lambda_k$ denotes the $k$-dimensional Lebesgue measure and $\lVert . \rVert$ is the distance to the next integer (see Section \ref{notation} for proper definitions and notations).
In their seminal article, Duffin and Schaeffer \cite{Duffin_Schaeffer_1941} constructed a (non-monotonic) function $\psi$ such that $ \sum_{q \in \N} \psi(q) = \infty$, but for almost every $\alpha$, \eqref{central_question} has only finitely many solutions. Their example makes heavy use of the fact that for given $x \in \Q$, there are many different pairs $(a,q) \in \Z^2$ such that $\frac{a}{q} = x$. To resolve this issue, Duffin and Schaeffer suggested to only consider reduced fractions to make this representation unique.
To formalize this, let 
\[\lVert qx\rVert' := \min_{\gcd(p,q) =1}\lvert qx -p \rvert\]
denote the best approximation by coprime fractions and for given $\psi: \N \to [0,\infty)$, consider the set

\[D(\psi) := \left\{\alpha \in [0,1]: \lVert q\alpha \rVert' \leq \psi(q) \text{ for i.m. }q \in \N\right\}.
\]
Consequently, Duffin and Schaeffer conjectured that for any $\psi: \N \to [0,\infty)$, we have
\[\lambda_1\left(D(\psi)\right)
   = \begin{cases}
       1, &\text{ if } \sum\limits_{q \in \N} \psi(q)\frac{\varphi(q)}{q} = \infty,\\
       0, &\text{ if } \sum\limits_{q \in \N}\psi(q)\frac{\varphi(q)}{q} < \infty.
   \end{cases}\]

This became one of the most famous open questions in Diophantine approximation and after contributions made in, e.g.,
\cite{Aistleitner_2014,Aistleitner_Lachmann_Munsch_Technau_Zafeiropoulos_2019,Beresnevich_Harman_Haynes_Velani_2013,Erdos_1970,Gallagher_1961,Haynes_Pollington_Velani_2012,PollingtonVaughan1990,Vaaler_1978}, it was finally proven by a breakthrough result of Koukoulopoulos and Maynard \cite{Koukoulopoulos_Maynard_2020}.

\subsection{The approximation theory in higher dimensions}

When looking at Diophantine approximation in dimension $k \geq 1$, there are several classical and well-studied measures of approximation quality considered. We will highlight the two topics that are likely the most extensively studied: One the one hand there is the simultaneous approximation where one considers the set

\[S_k(\psi) := \left\{\alpha \in [0,1]^k: \max\limits_{1 \leq i \leq k}\lVert q\alpha_i \rVert \leq \psi(q) \text{ for i.m. }q \in \N\right\},\]
and its coprime counterpart
\[D_k(\psi):= \left\{\alpha \in [0,1]^k: \max\limits_{1 \leq i \leq k}\lVert q\alpha_i \rVert'  \leq \psi(q) \text{ for i.m. }q \in \N\right\}.
\]
Another well-researched area concerns multiplicative Diophantine approximation, where one examines the set

\[S_k^{\times}(\psi) := \left\{\alpha \in [0,1]^k: \prod\limits_{i=1}^k\lVert q\alpha_i \rVert \leq \psi(q) \text{ for i.m. }q \in \N\right\},\]
and its coprime counterpart
\[D_k^{\times}(\psi):= \left\{\alpha \in [0,1]^k: \prod\limits_{i=1}^k\lVert q\alpha_i \rVert' \leq \psi(q) \text{ for i.m. }q \in \N\right\}.
\]
The case of simultaneous approximation is metrically well understood: A straightforward adaptation of Khintchine's Theorem shows that for monotonic $\psi$ and any $k \in \N$,
\[\lambda_k(S_k(\psi)) = \begin{cases}
       1, &\text{ if } \sum\limits_{q \in \N} \psi(q)^k = \infty,\\
       0, &\text{ if } \sum\limits_{q \in \N}\psi(q)^k < \infty.
       \end{cases}\]
Gallagher \cite{Gallagher_1962} proved that for $k \geq 2$, the monotonicity assumption can be removed, which is, considering the Duffin-Schaeffer counterexample, in contrast to the case $k = 1$.
In the coprime setup, Pollington and Vaughan \cite{PollingtonVaughan1990} proved that for 
any (not necessarily monotonic) function $\psi$ and $k \geq 2$,
\[\lambda_k(D_k(\psi)) = \begin{cases}
       1, &\text{ if } \sum\limits_{q \in \N} \left(\frac{\varphi(q)\psi(q)}{q}\right)^k= \infty,\\
       0, &\text{ if } \sum\limits_{q \in \N} \left(\frac{\varphi(q)\psi(q)}{q}\right)^k < \infty, 
       \end{cases}\]
       and by the much later result of Koukoulopoulos and Maynard \cite{Koukoulopoulos_Maynard_2020}, this also holds true in the case of $k = 1$.\\

We now direct our focus towards the investigation of multiplicative Diophantine approximation, a subject that has attracted considerable attention in recent years, largely due to its connection to a well-known unsolved problem: Littlewood's conjecture postulates that for any $\varepsilon > 0$ and $\psi_{\varepsilon}(q) := \frac{\varepsilon}{q}$, we have
$\lim_{\varepsilon \to 0} S_2^{\times}(\psi_{\varepsilon}) = \emptyset$.
Although Einsiedler, Katok, and Lindenstrauss \cite{Einsiedler2006} could show that the Hausdorff dimension of $\lim_{\varepsilon \to 0} S_2^{\times}(\psi_{\varepsilon})$ is $0$, the conjecture itself remains widely open. Clearly, the result in \cite{Einsiedler2006} implies that for almost every $(\alpha,\beta) \in [0,1]^2$, $\liminf_{n \to \infty} n\lVert n\alpha \rVert \lVert n\beta \rVert = 0$ (which can already be deduced from Khintchine's Theorem), but the rate of approximation can be improved in the multiplicative setup when allowing an exceptional set of (two-dimensional) Lebesgue measure $0$. For various adaptations of this result, we refer the reader to \cite{BadVel2011} and the references therein.
In the metric question, Gallagher \cite{Gallagher_1962} proved the following, assuming monotonicity of the approximation function\footnote{We remark that in the actual statement of Gallagher, there is no restriction on $\psi$ (clearly for $\psi(q) \geq 1/2$, the question becomes trivial), and $\log(1/\psi(q))$ is replaced by $\log(q)$. This is because in the monotonic case, the critical regime is $\log(q) \asymp \log(1/\psi(q))$, which is no longer true when $\psi(q)$ fails to be monotonic. The statement in the form given above was used already in several subsequent works such as \cite{Beresnevich_Haynes_Velani2013,Kristensen_2004,Wang_Yu1981}.
Actually, Beresnevich, Haynes, and Velani \cite{Beresnevich_Haynes_Velani2013} constructed a non-monotonic counterexample where for $k \geq 2$, $\sum\limits_{q \in \N} \psi(q)^k\log(q)^k = \infty$, but $\lambda_k(S_k^{\times}(\psi)) = 0$.}.

\begin{thm*}[Gallagher, 1962]
   Let $\psi: \N \to [0,1/2]$ be a \textbf{monotonically decreasing} function. Then for any $k \geq 1$, we have
   \[\lambda_k\left(S_k^{\times}(\psi)\right)
   = \begin{cases}
       1, &\text{ if } \sum\limits_{q \in \N} \psi(q) \log(1/\psi(q))^{k-1} = \infty,\\
       0, &\text{ if } \sum\limits_{q \in \N}\psi(q) \log(1/\psi(q))^{k-1} < \infty.
   \end{cases}
   \]
\end{thm*}

Roughly speaking, Gallagher's Theorem shows that for almost all pairs $(\alpha,\beta)$, we can surpass Littlewood's conjecture by approximately a factor of $\log(q)^2$, and this bound is optimal up to double-logarithmic terms. Furthermore, for each added variable, we almost surely gain another power of $\log(q)$.

In view of the counterexample of Duffin and Schaeffer \cite{Duffin_Schaeffer_1941} for a non-monotonic version of Khintchine's Theorem, the question arises whether monotonicity is a necessary assumption in Gallagher's Theorem.
Since in dimension $k=1$, the multiplicative and simultaneous setups coincide, i.e.
$S_1(\psi) = S_1^{\times}(\psi), D_1(\psi) = D_1^{\times}(\psi)$, 
we know that in the $1$-dimensional case, monotonicity cannot be removed from Gallagher's Theorem for $S_1^{\times}(\psi)$, but it is possible to remove it for $D_1^{\times}(\psi)$ when changing the divergence condition by a factor of $\frac{\varphi(q)}{q}$.
Thus, the open question reduces to the case where $k \geq 2$. Inspired by a different result of Gallagher \cite{Gallagher_1962} and Pollington-Vaughan \cite{PollingtonVaughan1990} for the simultaneous approximation, the following two conjectures (for coprime and non-coprime setup) were posed by Beresnevich, Haynes and Velani \cite{Beresnevich_Haynes_Velani2013}:

\begin{conj}\label{conj_normal}
   Let $\psi: \N \to [0,1/2]$ be an arbitrary function. Then for any $k \geq 2$, we have
   \[\lambda_k\left(S_k^{\times}(\psi)\right)
   = \begin{cases}
       1, &\text{ if } \sum\limits_{q \in \N} \psi(q) \log(1/\psi(q))^{k-1} = \infty,\\
       0, &\text{ if } \sum\limits_{q \in \N}\psi(q) \log(1/\psi(q))^{k-1} < \infty.
   \end{cases}
   \]
\end{conj}

\begin{conj}\label{conj_coprime}
       Let $\psi: \N \to [0,1/2]$ be an arbitrary function. Then for any $k \geq 1$, we have 
   \[\lambda_k\left(D_k^{\times}(\psi)\right)
   = \begin{cases}
       1, &\text{ if } \sum\limits_{q \in \N} \left(\frac{\varphi(q)}{q}\right)^k \psi(q) \log(1/\psi(q))^{k-1} = \infty,\\
       0, &\text{ if } \sum\limits_{q \in \N}\left(\frac{\varphi(q)}{q}\right)^k \psi(q) \log(1/\psi(q))^{k-1} < \infty.
   \end{cases}
   \]
\end{conj}

The main purpose of this article is to (dis-)prove Conjecture \ref{conj_coprime}:

\begin{thm}\label{main_thm}
       Let $\psi: \N \to [0,1/2]$ be an arbitrary function. Then for any $k \geq 1$, we have 
   \[\lambda_k\left(D_k^{\times}(\psi)\right)
   = \begin{cases}
       1, &\text{ if } \sum\limits_{q \in \N} \left(\frac{\varphi(q)}{q}\right)^k \psi(q)\log\left(\frac{q}{\varphi(q)\psi(q)}\right)^{k-1} = \infty,\\
       0, &\text{ if } \sum\limits_{q \in \N}\left(\frac{\varphi(q)}{q}\right)^k \psi(q)\log\left(\frac{q}{\varphi(q)\psi(q)}\right)^{k-1}< \infty.
   \end{cases}
   \]
\end{thm}

Clearly, Theorem \ref{main_thm} is in the vein of Conjecture \ref{conj_coprime}, since this is a Khintchine-type $0$-$1$-result that only takes into account the divergence property of a series solely depending on $\psi$. However, the condition is altered by a factor of $q/\varphi(q)$ in the log-powers and it is easy to construct functions $\psi$ such that this factor determines the convergence of the sum, thus strictly speaking, we disprove Conjecture \ref{conj_coprime}. At first glance, it may appear surprising that the correct criterion is not immediately evident. However, unlike the simultaneous case, the divergence property cannot be simply inferred from a sum of measures argument. This is because the underlying hyperbolic regions exhibit a complex (and partially overlapping) structure, in contrast to the hyper-cubes in the simultaneous counterpart. We will explain and illustrate how to come up with the correct quantities in Section \ref{conv_case}.\\

Prior to this article, asking for $\lVert q \alpha\rVert \in A_q$ for i.m. $q$ and some measurable sets $(A_q)_{q \in \N} \subseteq [0,1/2]^k$ in the metric setup required (to the best of the authors' knowledge) at least one of the following assumptions (see, e.g., \cite{Beresnevich_Haynes_Velani2013,Harman_1998} for a more comprehensive discussion):

\begin{itemize}
 \item The sequence of measures $(\lambda_k(A_q))_{q \in \N}$ is monotonically decreasing, or
    \item The shape of every $A_q$ is convex.
\end{itemize}

We provide a non-exhaustive list of examples illustrating the application of monotonicity in related contexts: The actual statement of Gallagher \cite{Gallagher_1962} allows general sets $(A_q)_{q \in \N}$ with the property
$(y_1,\ldots,y_k) \in A_q \Rightarrow (y_1',\ldots,y_k') \in A_q$ for all 
$0 \leq |y_i'| \leq y_i\;\forall 1 \leq i \leq k$, as long as $(\lambda_k(A_q))_{q \in \N}$ is monotonically decreasing. This includes both the cases considered in $S_k(\psi)$ and $S_k^{\times}(\psi)$. Under the same assumptions on $A_q$, Wang and Yu \cite{Wang_Yu1981} even established an asymptotic Schmidt-type formula for the number of solutions to $\{q \leq Q: \lVert q\alpha \rVert \in A_q\}$ for almost every $\alpha$ when $Q \to \infty$. The above-mentioned property was further relaxed in the works of Dodsen and Kristensen \cite{Dodson_Kristensen_2006,Kristensen_2004} where star-shaped domains with respect to $0$ were considered. Nevertheless, monotonicity remained a crucial assumption in the proofs of all mentioned results.

In the regime of convex shapes, the most famous example is the already mentioned work of Pollington and Vaughan \cite{PollingtonVaughan1990} where the corresponding $A_q$'s are $k$-dimensional hypercubes. The shape of $A_q$ can be altered, as long as the convexity property is preserved.
The usage of convex bodies helps to estimate the overlap since counting rational points in convex domains is much easier than in arbitrary shapes. For a detailed discussion, we refer the reader to the monograph of Harman \cite[Chapter 3]{Harman_1998}.\\

While the assumption of monotonicity is not required in Theorem \ref{main_thm}, the sets under examination are of hyperbolic shape, making them highly non-convex. This non-convexity significantly complicates the analysis of overlaps. We regard the capability to address hyperbolic shapes, which arise naturally within the context of multiplicative Diophantine approximation, as the primary technical novelty of Theorem \ref{main_thm}.\\

The new idea employed in this article is to establish not a classical overlap estimate, but to bound the correlation of functions $\int_{0}^1 \gamma_q(x)\gamma_r(x) \,\mathrm{d}x$ in terms of
$\left(\int_{0}^1 \gamma_q(x) \,\mathrm{d}x\right)\left(\int_{0}^1\gamma_r(x) \,\mathrm{d}x\right)$,
where $\gamma_q$ consists of small bumps in the shape of $1/x$ around rationals $a/q$ with $(a,q) = 1$, which is different from the simultaneous setting where the bumps are just indicators (or in the convex generalization, can be approximated well enough with indicators). After establishing a nuanced upper bound for $\gamma_q$, we approximate the corresponding bumps using a sequence of step-functions carefully chosen in the spirit of the Lebesgue integral construction. The estimation of controlling the correlation among these step-functions can be roughly traced back to the case of indicator functions, aligning with the framework of the Duffin-Schaeffer conjecture. By applying methods established in \cite{ABH2023, Koukoulopoulos_Maynard_2020} in a refined manner, we obtain the final ingredient.

\subsection{Applications and related work}

In \cite{Beresnevich_Haynes_Velani2013}, both a $0-1$-law  and a ``Duffin-Schaeffer-Theorem''-type result were proved. It was shown that assuming the sizes of $\psi(q)$ and $\frac{q}{\varphi(q)}$ do not have a too large (positive) correlation, i.e. if

\begin{equation}\label{DST_cond}
\limsup_{Q \to \infty} \sum_{q = 1}^{Q} \left(\frac{\varphi(q)}{q}\right)^k\psi(q)\left(\log (1/\psi(q))\right)^{k-1} 
\left(\sum_{q = 1}^{Q} \psi(q)\left(\log (1/\psi(q))\right)^{k-1}\right)^{-1} > 0,
\end{equation}
then Conjectures  \ref{conj_normal} and \ref{conj_coprime} are true. 
We note that in this regime, the conjectured divergence criteria from both Conjectures coincide with the one established in Theorem \ref{main_thm}, thus there is no contradiction in all statements being true. The following result is immediately implied by Theorem \ref{main_thm}, which makes progress towards Conjecture \ref{conj_normal} by demanding some extra-divergence.

\begin{cor}\label{weak_thm}
    Let $\psi: \N \to [0,1/2]$ be an arbitrary function
    such that \begin{equation}\label{extra_div}\sum\limits_{q \in \N} \left(\frac{\varphi(q)}{q}\right)^k \psi(q) \log\left(\frac{q}{\varphi(q)\psi(q)}\right)^{k-1} = \infty.\end{equation} Then
\[\lambda_k\left(S_k^{\times}(\psi)\right) = 1.\]
In particular, \eqref{extra_div} holds whenever
 \[\sum_{q \in \N} \frac{\psi(q)}{(\log \log q)^k }\log\left(\frac{\log\log q}{\psi(q)}\right)^{k-1}= \infty.\]
\end{cor}
For many applications, it is particularly interesting when $\psi$ is of the form $\psi(q)= \mathbbm{1}_{A}(q) \theta(q)$, where $A \subseteq \N$ and $\theta$ is a monotonically decreasing function. For various results in this context, we refer the interested reader again to Chapter 3 of the excellent book of Harman \cite{Harman_1998}.
Given an asymptotic density of $A$, i.e. a function $f:\N \to [0,1]$ such that
\[\frac{\#\{n \leq N: n \in A\}}{N} \sim f(N),\]
we ask whether the expected quality of approximation with $q \in A$ only loses a factor of $f(N)$ in comparison to Khinchine's Theorem. If we have some control of the average size of $\frac{q}{\varphi(q)}$ for $q \in A$, i.e.
\begin{equation}\label{well_behaved_A}
\limsup_{N \to \infty}\frac{\#\{q \leq N: q \in A, \frac{q}{\varphi(q)} \leq M\}}{\#\{q \leq N: q \in A\}} > 0,\end{equation}
then the result in \cite{Beresnevich_Haynes_Velani2013} is enough to obtain the precise information: To name just some easy examples, if $(p_n)_{n \in \N}$ is the sequence of all primes, then $f(N) = 1/\log N$ and $M = 2$ satisfy the above assumptions. 
We thus obtain the very simple criterion of whether $\sum_{q \in \N} \theta(q) = \infty$ (which coincides with the classical Khintchine Theorem).
Hence, we have 
$\lVert p\alpha \rVert\lVert p\beta \rVert \leq \frac{1}{p\log p}$ infinitely often, but 
$\lVert p\alpha \rVert\lVert p\beta \rVert \leq \frac{1}{p(\log p)^{1 + \varepsilon}}$ only finitely often.
If $A$ is a set of positive density, i.e. $f(N) = \delta > 0$, we can approximate with the same quality as in Gallagher's Theorem, that is, $\lVert q\alpha \rVert \lVert q\beta\rVert \leq \frac{1}{q (\log q)^2}$. If $A = \{b^n, n \in \N\}$ for some $b \geq 2$, the above shows that
\[\lVert b^n\alpha \rVert \lVert b^n\beta\rVert \leq \frac{1}{n^2(\log n)^2}\]
holds infinitely often, which is sharp up to triple-logarithmic factors.\\

However, such assertions cannot be inferred for arbitrary sequences $(a_n)_{n \in \mathbb{N}}$. Theorem \ref{main_thm} serves as a tool for deriving various statements in cases where only density estimates are available, without relying on number-theoretic conditions such as \eqref{well_behaved_A}. To illustrate this, we present Corollary \ref{littlewood_subsequence} below, which provides a good estimate on the required density of a given sequence $(a_n)_{n \in \N}$ such that Littlewood's conjecture is true along this sequence for almost every $(\alpha,\beta)$. More precisely, let $A = \{a_1 < a_2 < \ldots\}$ be an infinite set of positive integers and let
\[
LC(A) := \left\{(\alpha,\beta) \in [0,1]^2: \liminf_{n \to \infty} a_n \lVert a_n\alpha\rVert\lVert a_n\beta\rVert = 0 \right\},
\]
i.e. the set of irrationals such that Littlewood's conjecture is true along the sequence $(a_n)_{n \in \N}$ (we refer the interested reader to \cite[Section 1.2]{PVZZ2022} for related results).
The following is an immediate application of Corollary \ref{weak_thm}.

\begin{cor}\label{littlewood_subsequence}
For any $A \subseteq \N$, we have

\[\lambda_2(LC(A)) = \begin{cases}
    1, \text{ if } \quad \frac{\#\{n \leq N: n \in A\}}{N} \gg \frac{\log \log N}{(\log N)^2},\\
    0, \text{ if } \quad \frac{\#\{n \leq N: n \in A\}}{N} \ll \frac{1}{(\log N)^2}.
\end{cases}\]
\end{cor}

To the best of the authors' knowledge, this improves the currently best-known result that stems from a direct application of the Koukoulopoulos-Maynard Theorem, where the condition
\[ \frac{\#\{n \leq N: n \in A\}}{N} \gg \frac{1}{\log N},\]
is (up to triple-logarithms), the best bound known to guarantee $\lambda_2(LC(A)) = 1$.
Thus Theorem \ref{littlewood_subsequence} gives an improvement of a factor of (roughly) $\log N/(\log \log N)
$, with the quality of gain increasing by another factor of $\log N/(\log \log N)$ when we consider more variables, i.e. $\lVert n\alpha \rVert \lVert n\beta\rVert \lVert n\gamma \rVert$. One way to get rid of the double logarithms and therefore obtain a sharp statement would be to prove (at least some form of) Conjecture \ref{conj_normal}, a question that remains open.

\subsection{Hausdorff dimension analogue}

Many statements in metric Diophantine approximation have a Hausdorff measure equivalent: By assuming a fast decay of the approximation function $\psi$, the sets $S_k(\psi),S_k^{\times}(\psi),D_k(\psi),D_k^{\times}(\psi)$ are all of Lebesgue measure $0$; however, by a phenomenon now widely known as ``mass transference principle'', which was popularized in the seminal work of Beresnevich-Velani \cite{beresnevich_velani_2006}, one can immediately deduce a Hausdorff dimension version of Lebesgue measure theorems in many different setups. The Hausdorff version for $S_k(\psi)$ with $\psi$ monotone is the classical Jarník-Besicovitch-Theorem, and Beresnevich-Velani \cite{beresnevich_velani_2006} established\footnote{Strictly speaking, the result of  Beresnevich-Velani proved that the Duffin-Schaeffer Conjecture would imply a Hausdorff version of itself, thus the actual proof of the Hausdorff version was implied by the later work of Koukoulopoulos and Maynard in \cite{Koukoulopoulos_Maynard_2020}.} the analogue for $D_k(\psi)$, again without assuming monotonicity.
In the setting of multiplicative Diophantine approximation, the picture is not yet that clear, although progress was made in \cite{Beresnevich_Haynes_Velani2013,Bovey_Dodson_1978,Dodson_1991,Mumtaz_Simmons2018} for $S_k^{\times}(\psi)$ and assuming monotonicity, at least in the divergent case. The best result known so far is due to Hussain and Simmons \cite{Mumtaz_Simmons2018}:
For monotonic $\psi$ with $\psi(q) \to 0$, they showed that
$
\dim_H(S_k^{\times}(\psi)) =  (k-1) + \min\{d(\psi),1\}
$
where
\begin{equation}
\label{def_d_psi}
d(\psi) =  \inf\left\{s \in [0,1]: \sum_{q \in \N} q\left(\frac{\psi(q)}{q}\right)^{s} < \infty\right\}.
\end{equation}

As a Corollary of Theorem \ref{main_thm}, it turns out that the Hausdorff dimension statement remains unchanged for $S_k^{\times}(\psi)), D_k^{\times}(\psi))$ even upon removing the monotonicity assumption:

\begin{cor}\label{Hausdorff_thm}
    Let $\psi: \N \to [0,1/2]$ be an arbitrary function and let $d(\psi)$ be as in \eqref{def_d_psi}. Then we have
    \[\dim_H(S_k^{\times}(\psi)) =  \dim_H(D_k^{\times}(\psi))  = (k-1) + \min\{d(\psi),1\}.\]
\end{cor}
We note that Corollary \ref{Hausdorff_thm} is not only making a statement on the coprime approximations as considered in Conjecture \ref{conj_coprime}, but also about the non-coprime setting from Conjecture \ref{conj_normal}. Thus, as an immediate Corollary, we also get a Hausdorff version of Conjecture \ref{conj_normal}:

\begin{cor}
\label{Hausdorff_cor}
 Let $\psi: \N \to [0,1/2]$ be an arbitrary function and assume that \[  \sum\limits_{q \in \N} \psi(q) \log(1/\psi(q))^{k-1} = \infty.\]
 Then $\dim_H(S_k^{\times}(\psi)) = k$.
\end{cor}

\subsection{Open questions}\label{open_q}
Besides the obvious question of proving also Conjecture \ref{conj_normal}, there are natural generalizations resp. quantifications of Theorem \ref{main_thm}. 
One way to do so is to allow an inhomogeneous parameter $\gamma \in \mathbb{R}^k$ and ask the same questions about the sets

\[S_k^{\times}(\psi,\gamma) := \left\{\alpha \in [0,1]^k: \prod\limits_{i=1}^k\lVert q\alpha_i  - \gamma_i \rVert \leq \psi(q) \text{ for i.m. }q \in \N\right\}\]
and their coprime counterparts
\[D_k^{\times}(\psi,\gamma):= \left\{\alpha \in [0,1]^k: \prod\limits_{i=1}^k\lVert q\alpha_i -\gamma_i \rVert' \leq \psi(q) \text{ for i.m. }q \in \N\right\}.
\]

Assuming monotonicity of $\psi$, the inhomogeneous variant of Gallagher's result was proven for $k = 1$ by Sz\"{u}sz \cite{Szusz_1958} and by Chow and Technau \cite{Chow_2018,Chow_Technau_2020} for $k \geq 2$. The authors of the latter obtained actually a stronger result since they only need to have one metric parameter $\alpha_k$ assuming a mild Diophantine condition for $(\alpha_1,\ldots,\alpha_{k-1})$. However, all these proofs made use of the monotonicity of $\psi$, thus the statement remains open in the non-monotonic setup.
Recently a first result for non-monotonic $\psi$ was given by Yu \cite{Yu3}, albeit a strong form of regularity in the form of $\psi(q) = O((q\log q (\log \log q)^2)^{-1})$ is assumed.

By generalizing the Duffin-Schaeffer counterexample, Ramírez \cite{Ramirez_counter} showed that for any $\gamma$, the monotonicity assumption is necessary when $k = 1$, but as in the homogeneous setup, the counterexample does not generalize to higher dimensions. Therefore, the theorem of Chow-Technau \cite{Chow_Technau_2020} or a coprime analogue in the form of Theorem \ref{main_thm} in the inhomogeneous setup might still be true without monotonicity assumption; however, the proof given in this article relies on the proven (homogeneous) Duffin-Schaeffer conjecture and the inhomogeneous counterpart is still a widely open question, for recent progress see \cite{Allen_Ramirez,Hauke,Yu1,Yu2}. In this context, the current methods do not offer a viable approach to addressing these questions until a deeper understanding of the one-dimensional inhomogeneous theory is attained. We also remark that Theorem \ref{main_thm} does not yield any information about the set $D_k^{\times}(\psi)$ other than having Lebesque measure $1$ or $0$. In particular, it does not provide a Diophantine condition that ensures $\alpha \in \R^k$ to lie in $D_k^{\times}(\psi)$. There is no possibility to provide such a ``fibered'' result as in \cite{Chow_Technau_2020} with the metric tools applied in this article, and the approach in \cite{Chow_Technau_2020} does not carry over to the non-monotonic setup.
For fibered results in the well-studied regime of Littlewood's conjecture, we refer the reader to, e.g., \cite{Chow_Technau_2023,Chow_Zafeiropoulos_2021,Pollington_Velani_2000}.

In the theory of Hausdorff measures, the question about Hausdorff dimensions is coarse enough to absorb all $\log$-factors and $\frac{q}{\varphi(q)}$-factors that appear in the Lebesgue measure statements. However, when we consider arbitrary dimension functions $f$ that are not necessarily of the form $f(r) = r^s$, the picture is different. Here, a difference between $H^f(D_k^{\times}(\psi))$ and $H^f(S_k^{\times}(\psi))$ for specific $f$ seems plausible and the correct thresholds remain open.

Finally, another natural open question is to ask about an asymptotic formula that counts the number of coprime solutions up to a given threshold. This was successfully established by Aistleitner, Borda, and the second-named author in \cite{ABH2023} for the case $k = 1$, but is yet unclear for $k \geq 2$. Achieving this goal necessitates a deeper understanding of the variance estimate introduced in this article. Furthermore, and probably even more challenging, one cannot afford to truncate the sets $A_q$ in a way that loses a constant factor of the corresponding measures, a procedure that is applied at various occasions in the proof of Theorem \ref{main_thm}.

\subsection{Notation}\label{notation}
Throughout the text, we use the standard O- and o- notations,
 as well as the Vinogradov notations $\ll,\gg$. If $f \ll g$ and $g \ll f$, we write $f \asymp g$ or $f = \Theta(g)$.
 For $x \in \mathbb{R}$, $\lVert x\rVert := \min\limits_{n \in \mathbb{Z}} \lvert x - n\rvert$ denotes the distance to the nearest integer. For $a,q \in \Z$, we write $(a,q):= \gcd(a,q)$. For $q \in \N$, we denote $\Z_q:= \{ 0, \ldots, q-1 \}$ and $\Z_q^{*}:= \{ a \in \Z_q: (a,q)=1 \}$. For $k$ being an integer, we denote the $k$-dimensional Lebesgue-measure as $\lambda_k$. Given two positive real numbers $x,y$, we write $x \land y := \min \left\{x,y \right\}$. Furthermore, for any $k \in \N$, we define $x \log^k (1/x) \mid_{x = 0} := 0.$\\

\section{Proof of the main result}

\subsection{The convergent case and a discussion on the correct divergence condition}\label{conv_case}

Before we prove the convergent case (which is easier than the divergent case, but still not as trivial as in most other scenarios), we discuss the particular divergence condition of Theorem \ref{main_thm} and its deviation from the conjectured threshold in \cite{Beresnevich_Haynes_Velani2013}. We will focus on the case $k = 2$ here since in the case $k=1$ the conditions coincide and for $k\geq 3$, an analogous phenomenon to $k =2$ takes place.
Starting with the non-coprime case, given some function $\psi:\N \to [0,1/2]$, let
\[A_q = A_q(\psi) := \left\{(\alpha,\beta) \in [0,1]^2: \lVert q\alpha \rVert \lVert q\beta\rVert \leq \psi(q)\right\}.\] Note that we can write $S_2^{\times}(\psi) = \limsup_{q \to \infty}A_q$.
We observe that $A_q$ consists of a regular pattern of $q^2$ "star-shaped" forms with their origins at the rational points $(a/q,b/q)$ (see Figure \ref{Figure_no_coprime}). It is easy to check that we have a partition into

\begin{equation}\label{part_non_coprime}A_q = 
\bigcup\limits_{(a,b) \in \mathbb{Z}_q}^{\boldsymbol{\cdot}} A_{q,a,b}
\end{equation}
where
\[A_{q,a,b} = \left\{(\alpha,\beta) \in [0,1]^2: \left\lvert \alpha - \frac{a}{q} \right\rvert\left\lvert \beta - \frac{b}{q} \right\rvert \leq \frac{\psi(q)}{q^2}, \left\lvert \alpha - \frac{a}{q} \right\rvert \leq \frac{1}{2q}, \left\lvert \alpha - \frac{b}{q} \right\rvert \leq \frac{1}{2q}\right\}.\]
By elementary calculus, we obtain for any $(a,b) \in \mathbb{Z}_q^{2}$ that 
\[\lambda_2(A_{q,a,b}) \asymp \frac{\psi(q)}{q^2}\log\left(\frac{1}{\psi(q)}\right)\]
with an absolute implied constant. Thus by \eqref{part_non_coprime} we obtain $\lambda_2(A_q) \asymp \psi(q)\log\left(\frac{1}{\psi(q)}\right)$, which coincides with the divergence condition in Conjecture \ref{conj_normal} and in Gallagher's Theorem.
\begin{figure}[H]
\begin{center}
\includegraphics[scale=1, trim=130 530 150 70]{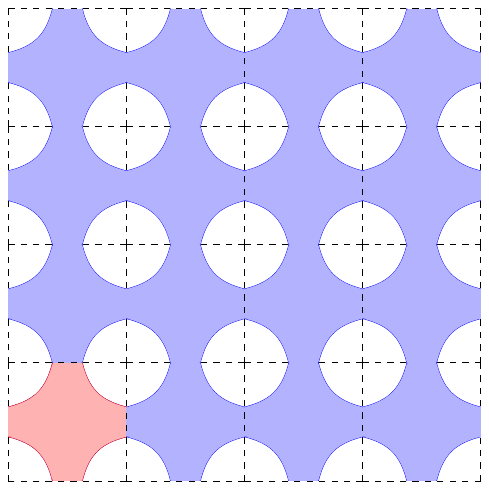}
\caption{An illustration of the set $A_4$ (the non-coprime setting). The red area corresponds to $A_{4,1,1}$. Note that for illustration purposes, the lower-left corner has the coordinate $ ( 1/8, 1/8 )$.}
\label{Figure_no_coprime}
\end{center}
\end{figure}
Changing the attention now to the coprime setting, the question about the Lebesgue measure of
$A_q' = A_q'(\psi) := \left\{(\alpha,\beta) \in [0,1]^2: \lVert q\alpha \rVert' \lVert q\beta\rVert' \leq \psi(q)\right\}$ becomes non-trivial:
One is tempted to apply the same partition argument as in the non-coprime setup in the form of
\begin{equation}
\label{part_coprime}
\bigcup\limits_{(a,b) \in \mathbb{Z}_q^*}^{.} A_{q,a,b} =: A_q'',
\end{equation}
and one would assume $A_q' = A_q''$. This would give us
$\lambda_2(A_q') \asymp \frac{\varphi(q)^2}{q^2}\log\left(\frac{1}{\psi(q)}\right)$ which would coincide with the measures that lead to the condition in Conjecture \ref{conj_coprime}.

However, \eqref{part_coprime} is not true: Assuming the contrary, this would imply that for $(a,b) \in \mathbb{Z}_q$ such that $(a,q) > 1$ or $(b,q) > 1$, the sets
$A_q'$ and $B_{a,b,q} := \left\{(\alpha,\beta) \in [0,1]^2: \left\lvert \alpha - \frac{a}{q} \right\rvert \leq \frac{1}{2q}, \left\lvert \alpha - \frac{b}{q} \right\rvert \leq \frac{1}{2q}\right\}$ are disjoint, but this is wrong: Assuming, e.g., $(a-1,q) = 1, (b,q) = 1$, then a very good approximation of $\beta$ by $b/q$ implies that even for $\alpha$ close to $\frac{a-1}{q}$ (which implies $\left\lvert \alpha - \frac{a}{q} \right\rvert \geq \frac{1}{2q}$), we have $\left\lvert \alpha - \frac{a}{q} \right\rvert\left\lvert \beta - \frac{b}{q} \right\rvert \leq \frac{\psi(q)}{q^2}$. To put this in simple words, the ``arm'' of the ``star'' with its center at $((a-1)/q,b/q)$ ``reaches'' into $B_{a,b,q}$ and if 
$(a,b) \notin {(\mathbb{Z}_q^{*})}^2$, then this part of the arm is in $A_q'$, but not in $A_q''$ (see Figure \ref{Figure_coprime}). Actually, the ``arm'' reaches as far as half of the distance to another coprime element in the same direction. It can be verified by a maximization problem (which is reflected in the equality case of the arithmetic-geometric mean inequality, which appears in the proof of Proposition \ref{prop_conv_DS}) that the extremal configuration happens when the coprime elements are equally spaced among $\mathbb{Z}_q$, which leads to each such ``arm'' to be of length $\asymp \frac{q}{\varphi(q)}$. Computing the mass of one such ``star'' with ``arm lengths'' $\asymp \frac{q}{\varphi(q)}$ yields the measure
$\frac{\psi(q)}{q^2}\log\left(\frac{\frac{q}{\varphi(q)}}{\psi(q)}\right)$ and since there are $\varphi(q)^2$ such coprime ``stars'', we obtain the divergence condition from Theorem \ref{main_thm}. 

Having the heuristic explanation established, we turn to the actual proof of the convergent case of Theorem \ref{main_thm}.
\begin{figure}[H]
\begin{center}
\includegraphics[scale=1, trim=150 540 150 70]{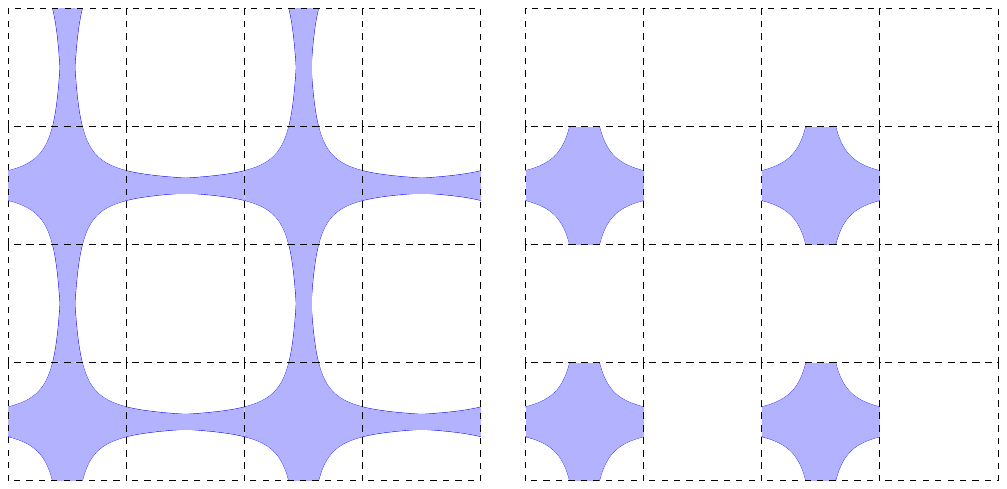}
\caption{To the left, we see $A_4'$, which is the proper set in the coprime setting. To the right, we have $A_4''$, which would give the conjectured divergence criterion in \cite{Beresnevich_Haynes_Velani2013}. We note that in both figures, the lower left corner has the coordinate $(1/8,1/8)$.}
\label{Figure_coprime}
\end{center}
\end{figure}
\begin{proposition}
\label{prop_conv_DS}  
Let $k \geq 1$ and $\psi: \N \rightarrow [0,1/2]$ be a function such that
\begin{equation*}
    \sum_{q \in \N} \left( \frac{\varphi(q)}{q} \right)^{k} \psi(q) \log\left( \frac{q}{\psi(q) \varphi(q)}\right)^{k-1} < \infty.
\end{equation*}
Then $ \lambda_k(D_k^{\times}(\psi)) = 0$.
\end{proposition}
\begin{proof}
Defining $\Psi(q):= \psi(q)\left(\frac{\varphi(q)}{q} \right)^{k-1}$, we observe that
\begin{align*}
    \# \left\{ q \in \N : \prod_{i=1}^k || q \alpha_i ||' \leq \psi(q) \right\} & \leq \# \left\{ q \in \N : ||q \alpha_k ||' \leq \frac{\psi(q)}{\prod_{i=1}^{k-1} || q \alpha_i||'}, \forall i \in \{ 1, \ldots, k-1 \} : || q \alpha_i||' \geq \Psi(q) \right\} \\
    & \qquad + \# \left\{ q \in \N : \exists i \in \{1, \ldots, k-1 \} : || q \alpha_i ||' \leq \Psi(q)  \right\} \\
    & \leq \# \left\{ q \in \N : ||q \alpha_k ||' \leq \psi(q) \prod_{i=1}^{k-1} \left[  \frac{1}{|| q \alpha_i||'} \land \frac{1}{\Psi(q)}\right] \right\} \\
    & \qquad + \sum_{i=1}^{k-1} \# \left\{ q \in \N: || q \alpha_i ||' \leq \Psi(q)  \right\} \\
    &=: M_1(q) + M_2(q).
\end{align*}
We observe that $M_2(q)$ is almost surely finite by the convergence part of the Koukoulopoulos-Maynard-Theorem, since by assumption we have 
\[
\sum_{q \in \N} \frac{\varphi(q)}{q} \Psi(q) = \sum_{q \in \N } \psi(q) \left( \frac{\varphi(q)}{q} \right)^{k} < \infty.
\]
Thus, we are left with $M_1(q)$, which is almost surely finite by the convergence part of the Koukoulopoulos-Maynard theorem, provided that
\begin{equation}
\label{Eq_conv_psi}
\sum_{q \in \N} \frac{\varphi(q)}{q} \psi(q)  \psi_{\alpha}(q) < \infty 
\end{equation}
holds for almost all $\alpha \in [0,1]^{k-1}$, where
\[
\psi_{\alpha}(q) := \prod_{i=1}^{k-1} \left[ \frac{1}{|| q \alpha_i ||'} \land \frac{1}{\Psi(q)} \right].
\]
The condition in \eqref{Eq_conv_psi} is certainly true if
\[
\int_{[0,1]^{k-1}} \sum_{q \in \N} \frac{\varphi(q)}{q} \psi(q) \psi_{\alpha}(q) \,\mathrm{d}\alpha = \sum_{q \in \N} \frac{\varphi(q)}{q} \psi(q)  \int_{[0,1]^{k-1}}  \psi_{\alpha}(q) \,\mathrm{d}\alpha < \infty,
\]
where we are allowed to interchange integral and summation by Fubini's theorem for non-negative integrands. We further get
\begin{align*}
\int_{[0,1]^{k-1}}  \psi_{\alpha}(q) \,\mathrm{d}\alpha = \left( \int_0^1 \tilde{\psi}_{\beta}(q) \,\mathrm{d}\beta \right)^{k-1},
\end{align*}
with $\tilde{\psi}_{\beta}(q) := \frac{1}{|| q \beta ||'} \land \frac{1}{\Psi(q)}$ for $\beta \in [0,1]$. In the following, we compute an upper bound for the expected value of $\tilde{\psi}_{\beta}(q)$, where we obtain
\begin{align*}
\int_{0}^{1} \tilde{\psi}_{\beta}(q) \,\mathrm{d}\beta
= \sum_{b=0}^{q-1} \int_{b/q}^{(b+1)/q}  \left( \max_{a \in \Z_q^*}\frac{1}{|q \beta - a|} \land \frac{1}{\Psi(q)} \right) \,\mathrm{d}\beta =:I_q.
\end{align*}
For any $ b \in \Z_q$, we introduce the following sets
\[
J_b := \left( \frac{b}{q} - \frac{1}{2q}, \frac{b}{q} + \frac{1}{2q}\right)  \setminus \left\{ \frac{b}{q}\right\} = \left( \frac{b}{q} - \frac{1}{2q}, \frac{b}{q} \right) \cup \left(\frac{b}{q}, \frac{b}{q} +  \frac{1}{2q} \right) =: J_b' \cup J_b''.
\]
Then, $\arg \min_{c \in \Z_q^*} | q \beta - c| $ is constant on $J_b'$ and $J_b''$ respectively. This allows us to define the following covering of $\Z_q$ by
\[
A_a := \left \{ b \in \Z_q : a = \arg \min_{c \in \Z_q^*} | q \beta - c |, \forall \beta \in J_b' \right\} \cup \left \{ b \in \Z_q : a = \arg \min_{c \in \Z_q^*} | q \beta - c |, \forall \beta \in J_b'' \right\}, \quad a \in \Z_q^*.
\]
Note that any $b \in \Z_q$ can be contained in at most two different sets $A_a$, which leads to
\begin{align*}
    I_q &= \sum_{  b=0 }^{q-1} \int_{b/q - 1/2q}^{b/q + 1/2q} \left( \max_{a \in \Z_q^* } \frac{1}{| q \beta - a|} \land \frac{1}{\Psi(q)}\right) \,\mathrm{d}\beta \\
    & \leq 2 \sum_{ \substack{a=0 \\ (a,q)=1}}^{q-1} \sum_{ b \in A_a } \int_{J_b}  \frac{1}{| q \beta - a|} \land \frac{1}{\Psi(q)} \,\mathrm{d}\beta. 
\end{align*}
Now we observe that $A_a \subseteq [a - f_1(a), a + f_2(a)]$, where the $f_i(a) \in \N$ are such that 
\[ 
\sum_{\substack{a=0 \\ (a,q)=1}}^{q-1}  f_1(a) + f_2(a) \leq 2q. \]
Therefore, we obtain
\begin{align*}
    I_q & \leq \frac{1}{q} \sum_{(a,q)=1} \int_{- f_1(a)}^{f_2(a)} \frac{1}{| \beta| } \land \frac{1}{\Psi(q)} \,\mathrm{d}\beta \\
    & \ll \frac{1}{q} \sum_{(a,q)=1} \left[ \log\left(\frac{f_1(a)}{\Psi(q)}\right) + \log\left(\frac{f_2(a)}{\Psi(q)}\right) \right]\\
    & \ll \frac{\varphi(q)}{q} \log \left(\left( \frac{q}{\varphi(q)} \right)^k \frac{1}{\psi(q)} \right) \\
    & \ll \frac{\varphi(q)}{q} \log \left( \frac{q}{\varphi(q) \psi(q)} \right),
\end{align*}
where the estimate in the second last line uses the arithmetic-geometric mean inequality. This implies
\[
\sum_{q \in \N} \frac{\varphi(q)}{q} \psi(q)  \int_{[0,1]^{k-1}}  \psi_{\alpha}(q) \,\mathrm{d} \alpha \ll 
\sum_{q \in \N} \left( \frac{\varphi(q)}{q} \right)^{k} \psi(q) \log \left( \frac{q}{\varphi(q) \psi(q) } \right)^{k-1} < \infty,
\]
as desired.
\end{proof}

\subsection{The divergent case}
\subsubsection{Proof of Theorem \ref{main_thm} assuming a variance estimate}
\label{SubsubSec_strat}

Let $k \geq 2$ be fixed and recall that, under the assumption of
\begin{equation}
\label{Eq_div_sum_k}
\sum_{q \in \N} \left( \frac{\varphi(q)}{q} \right)^{k} \psi(q) \log \left( \frac{q}{\varphi(q)\psi(q)}\right)^{k-1} = \infty,
\end{equation}
we aim to show that for almost all $\alpha=(\alpha_1, \ldots, \alpha_k) \in \R^k$, there are infinitely many $q \in \N$ such that
\[
|| q \alpha_1 ||' || q \alpha_2 ||' \cdots || q \alpha_k||' \leq \psi(q).
\]
For the quantity in \eqref{Eq_div_sum_k}, for any fixed small $\rho > 0$, we have at least one of the two following:
\begin{align*}
 \sum_{\substack{q \in \N \\ \frac{\varphi(q)}{q} \leq \rho }} \left( \frac{\varphi(q)}{q} \right)^{k} \psi(q) \log \left( \frac{q}{\psi(q)\varphi(q)}  \right)^{k-1} = \infty \quad \text{or} \quad  \sum_{\substack{q \in \N \\ \frac{\varphi(q)}{q} > \rho }} \left( \frac{\varphi(q)}{q} \right)^{k} \psi(q) \log \left( \frac{q}{\psi(q)\varphi(q)} \right)^{k-1} = \infty.  
\end{align*}
Suppose first the latter to be true. Then $\tilde{\psi}(q) := \psi(q)\mathds{1}_{[\frac{\varphi(q)}{q} > \rho ]}$ satisfies \eqref{DST_cond}. Thus by \cite[Theorem 2]{Beresnevich_Haynes_Velani2013}, it follows immediately that $\lambda_k( D_k^{\times}(\psi)) = 1$. Thus, we may restrict ourselves to the case where for every $q$ in the support of $\psi$, it holds that
\begin{equation}
\label{Eq_cond_rho}
\frac{\varphi(q)}{q} \leq \rho
\end{equation}
for some small $\rho > 0$ that is specified below. We will show that almost surely,
\begin{equation}\label{divided_through}
\# \left\{ q \in \N : ||q \alpha_1 ||' \leq \psi_{\alpha}(q) \right\}= \infty,
\end{equation}
where $ \psi_{\alpha}(q) := \frac{\psi(q)}{\prod_{i=2}^{k} || q \alpha_i ||'}$. By the Koukoulopoulos-Maynard-Theorem, \eqref{divided_through} is true whenever
\begin{equation*}
\sum_{q \leq Q } \frac{\varphi(q)}{q}\psi_{\mathbf{\alpha}}(q) = \infty.
\end{equation*}
Defining
\begin{equation}
\label{Eq_def_rho}
\rho = \rho(k):= \frac{1}{4^{k-1}}
\end{equation}
and assuming \eqref{Eq_cond_rho}, we obtain $g_k(q):= \left( \frac{\psi(q) \varphi(q)}{q} \right)^{1/(k-1)} \leq 1/2$ for all $q \in \N$. Using this, a straightforward computation shows that
\begin{equation}\label{eq_geq_gamma}\frac{\varphi(q)}{q}\psi_{\mathbf{\alpha}}(q) \geq \gamma_q^{(k)}(\alpha),\end{equation}
where
\begin{equation*}
\gamma_q^{(k)}(x_1, \ldots, x_{k-1}) := \prod_{i=1}^{k-1}\sum_{(a,q)=1} f_q^{(k)}(q x_i -a),
\end{equation*}
with
\begin{equation*}
f_q^{(k)}(y) = \begin{cases}
    1, & \text{if} \quad |y| \leq g_k(q),\\
    \frac{g_k(q)}{|y|}, & \text{if} \quad |y| \in \left[ g_k(q), g_k(q)^{1/2} \right],\\
    0, & \text{else}.
\end{cases}
\end{equation*}

By elementary calculus and Fubini's Theorem, we obtain

\begin{equation}
    \label{divergent_expected}
\int_{[0,1]^{k-1}} \sum_{q \in \N} \gamma_q^{(k)}(\alpha) \,\mathrm{d} \alpha \gg 
\sum_{q \in \N} \left( \frac{\varphi(q)}{q} \right)^{k} \psi(q) \log \left( \frac{q}{\psi(q)\varphi(q)} \right)^{k-1} = \infty,
\end{equation}
where the last equality is assumption \eqref{Eq_div_sum_k}.
Next, we make use of the following variance estimate for $\gamma_q^{(k)}$, which is the remaining key ingredient in the proof of Theorem \ref{main_thm}:

\begin{proposition}
\label{prop_var1}
Let $\gamma_q^{(k)}$ be defined as above. Then for all $\delta > 0$, there exists a $Q(\delta) \in \N$ such that for all $Q \geq Q(\delta)$ we have
\[
\sum_{q,r \leq Q} \int_{[0,1]^{k-1}} \gamma_q^{(k)}(x) \gamma_r^{(k)}(x) \,\mathrm{d}x
\leq (1 + \delta) \left( \sum_{q \leq Q} \int_{[0,1]^{k-1}} \gamma_q^{(k)}(x) \,\mathrm{d}x \right)^2.
\]
\end{proposition}

Assuming Proposition \ref{prop_var1} (which is proven down below) to hold, we finish the proof of Theorem \ref{main_thm} in the following way by a standard argument:
Using \eqref{eq_geq_gamma}, it suffices to show that for almost all
$\alpha= (\alpha_1, \ldots, \alpha_{k-1}) \in \R^{k-1}$, we have

\begin{equation}
    \label{div_of_gammas}
S_{\alpha} :=  \sum_{q \in \N } \gamma_q^{(k)}(\alpha) = \infty.\end{equation}
For $c > 0$, let
$A_c := \{\alpha \in [0,1]^{k-1}: S_{\alpha} > c\}, B_c = [0,1]\setminus A_c$ and
$A_{\infty} = \lim_{c \to \infty}A_c$. Assuming to the contrary of \eqref{div_of_gammas} that
$\lambda_{k-1}(A_{\infty}) < 1$, by continuity of measures, there are $C> 0,\varepsilon> 0$ such that
$\lambda_{k-1}(B_C) > \varepsilon$. We write 
\[
S_{\alpha}^Q := \sum_{q \leq Q}\gamma^{(k)}_q(\alpha)
\] 
and we note that 
\[
\lim_{Q \rightarrow \infty} \int_{[0,1]^{k-1}} S_{\alpha}^Q \,\mathrm{d} \alpha = \infty
\]
by \eqref{divergent_expected}.
By Proposition \ref{prop_var1}, for any $\delta > 0$, we find a $Q=Q(\delta)$ such that
\[
\int_{[0,1]^{k-1}} \left[S_{\alpha}^Q\right]^2\,\mathrm{d} \alpha -\left(\int_{[0,1]^{k-1}} S_{\alpha}^Q \,\mathrm{d} \alpha\right)^2 \leq \delta \left(\int_{[0,1]^{k-1}} S_{\alpha}^Q \,\mathrm{d} \alpha\right)^2,
\]
for all $Q \geq Q(\delta)$. We choose $\delta = \frac{\varepsilon}{4}$ and apply Chebyshev's inequality which yields
\[
\lambda_{k-1}\left(\left\{\alpha \in [0,1]^{k-1}: \left\lvert S_{\alpha}^Q - \int_{[0,1]^{k-1}} S_{\alpha}^Q \,\mathrm{d} \alpha \right\rvert > \frac{1}{2}\int_{[0,1]^{k-1}} S_{\alpha}^Q \,\mathrm{d} \alpha \right\}\right) \leq 4 \delta,
\]
provided that $Q \geq Q(\delta)$. For such $Q$, we have
$S_{\alpha}^Q > \frac{1}{2}\int_{[0,1]^{k-1}} S_{\alpha}^Q \,\mathrm{d} \alpha$ on a set of Lebesgue measure of at least $1 - 4\delta$. Choosing $Q$ large enough such that
$\int_{[0,1]^{k-1}} S_{\alpha}^Q \,\mathrm{d} \alpha > 2C$ implies $\lambda_{k-1}(B_C) \leq 4\delta < \varepsilon$, a contradiction.

\subsubsection{The variance estimate}

\begin{proof}[Proof of Proposition \ref{prop_var1}]
Let $\varepsilon > 0$ and define 
\[
f_{q,\varepsilon}^{(k)}(y):=  \begin{cases}
1 , & \text{ if } 0 \leq |y| \leq g_k(q) \left( 1 + \varepsilon \log \left(g_k(q)^{-1}\right) \right), \\
f_q^{(k)}(y),  & \text{ otherwise}.
\end{cases}
\]
Moreover, for $x \in \R^{k-1}$ we set $\gamma_{q,\varepsilon}^{(k)}(x):= \prod_{i=1}^{k-1} \sum_{(a,q)=1} f_{q,\varepsilon}^{(k)}( qx_i -a)$ and recall that $g_k(q)= \left( \frac{\psi(q) \varphi(q)}{q} \right)^{1/(k-1)}$. Clearly, for all $y \in \R$ we have $f_{q}^{(k)}(y) \leq f_{q, \varepsilon}^{(k)}(y)$ and hence we get the estimate
\begin{align*}
\sum_{q,r \leq Q} \int_{[0,1]^{k-1}} \gamma_q^{(k)}(x) \gamma_r^{(k)}(x)\,\mathrm{d}x & \leq \sum_{q,r \leq Q } \left( \sum_{(a,q)=1} \sum_{(b,r)=1} \int_0^1 f_{q, \varepsilon}^{(k)}(qy-a) f_{r, \varepsilon}^{(k)}(ry-b) \,\mathrm{d}y \right)^{k-1}.
\end{align*}
For every $K \in \N$ and every $q \in \N$, we can choose a finite sequence of points
\[g_k(q) \left( 1 + \varepsilon \log \left( g_k(q)^{-1} \right) \right) \leq x_{q,0} \leq  \ldots \leq  x_{q,K-1} \leq 2 g_k(q)^{1/2}\] such that
\[
f_{q,\varepsilon}^{(k)}(y) \leq  \frac{1}{K} \sum_{i=0}^{K-1} \mathds{1}_{[-x_{q,i}, x_{q,i}]}(x), \qquad 
f_{q,\varepsilon}^{(k)}(y)= \lim_{K \rightarrow  \infty} \frac{1}{K} \sum_{i=0}^{K-1} \mathds{1}_{[-x_{q,i}, x_{q,i}]}(x).
\]

\begin{figure}[ht!]
    \centering
\begin{tikzpicture}[xscale=2, yscale=1.2]
    % Koordinatensystem
    \draw[<-] (-4,0.5) -- (0,0.5) node[left] {};
    \draw[->] (0,0.5) -- (4,0.5) node[left] {};
    \draw[->] (0,0.2) -- (0,5) node[above] {};

    % Funktion f_q(x)
    \draw[domain=1:5/4.5,smooth,variable=\x,red] plot ({\x},{ 5/\x -0.5} ) node[right] {};
    \draw[domain=5/4.5:3.5,smooth,variable=\x,violet] plot ({\x},{ 5/\x -0.5} ) node[right] {};
    \draw[domain=-3.5:-5/4.5,smooth,variable=\x,violet] plot ({\x},{ 5/(-\x)-0.5} ) node[right] {};
    \draw[domain=-5/4.5:-1,smooth,variable=\x,red] plot ({\x},{ 5/(-\x)-0.5} ) node[right] {};
    \draw[domain=-1:1,smooth,variable=\x,violet] plot ({\x},{ 4.5} ) node[right] {};
    \draw[dashed, color=violet] (-3.5,5/3.5-0.5)--(-3.5,5/3.5-1);
    \draw[dashed, color=violet] (3.5,5/3.5-0.5)--(3.5,5/3.5-1);

    % Funktion f_q^{\varepsilon}(x)
    %linker Ast
    \draw[domain=-5/4.5:-1,smooth,variable=\x,blue] plot ({\x},{ 4.5} ) node[right] {};
    \draw[domain=1:5/4.5,smooth,variable=\x,blue] plot ({\x},{ 4.5} ) node[right] {};
    \draw[dashed, color=blue] (-5/4.5,4.5) -- (-5/4.5,4);
    \node at (-5/4.5,4.5) [circle,fill,blue,inner sep=1pt] {};
    %\node at (-5/4,3.5) [circle,blue,inner sep=1pt] {};

    %rechter Ast
    \draw[dashed, color=blue] (5/4.5,4.5) -- (5/4.5,4);
    \node at (5/4.5,4.5) [circle,fill,blue,inner sep=1pt] {};
    
    % Funktion \frak{1}{K} \sum_{k=0}^{K-1} \mathds{1}_{-x_k. x_k}(x)
    % rechter Ast

    \draw (-5/4.5,4.5)--(-1.2,4.5);
    \draw (5/4.5,4.5)--(1.2,4.5);
    
     \draw[dashed] (-1.2,4.5)--(-1.2,4);
     \draw[dashed] (1.2,4.5)--(1.2,4);

    \foreach \x in {1.5,2,...,4}
    \def\y{\x}
    \draw (0,\x) -- (5/\y,\x);

    \foreach \x in {1.5,2,...,4}
    \def\y{\x-0.5}
    \draw[dashed] (5/\x,\x) -- (5/\x,\y);
    
    \draw (3.7, 1) -- (0,1);
    \draw[dashed] (3.7,1) -- (3.7,0.5);

     \draw (-3.7, 1) -- (0,1);
    \draw[dashed] (-3.7,1) -- (-3.7,0.5);

    % linker Ast

    \foreach \x in {-1.5,-2,...,-4}
    \def\y{\x}
    \draw (0,-\x) -- (5/\y,-\x);

    \foreach \x in {-1.5,-2,...,-4}
    \def\y{\x-0.5}
    \draw[dashed] (5/\x,-\x) -- (5/\x,-\y);
    
     \draw (1,0.5) -- (1,0.4) node[below] {$g_{k}(q)$};
      \draw (3.5,0.5) -- (3.5,0.4) node[below] {$g_{k}(q)^{1/2}$};
     \draw (-3.7,0.5) -- (-3.7,0.4) node[below] {$-x_{q,K-1}$};
\end{tikzpicture}
\caption{The picture illustrates the relation of $f_q^{(k)}(y)$, $f_{q,\varepsilon}^{(k)}(y) $ and $\frac{1}{K} \sum_{i=0}^{K-1} \mathds{1}_{[-x_{q,i}, x_{q,i}]}(y) $.}
\label{drawing}
\end{figure}
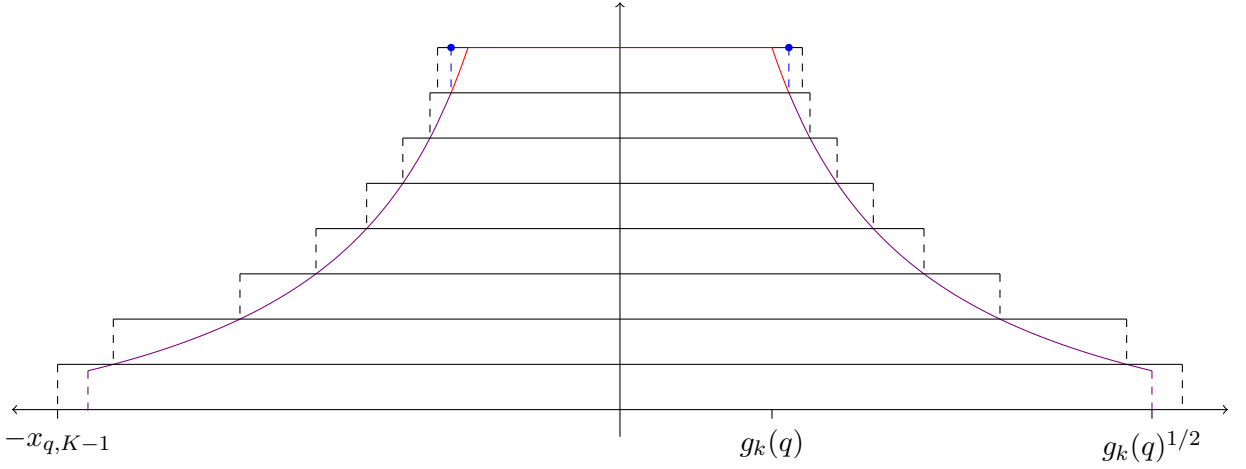

This leads to
\begin{align*}
 &\sum_{q,r \leq Q } \left( \sum_{(a,q)=1} \sum_{(b,r)=1} \int_0^1 f_{q,\varepsilon}^{(k)}(qy-a) f_{r,\varepsilon}^{(k)}(ry-b) \,\mathrm{d}y \right)^{k-1} \\
 \leq & \sum_{q,r \leq Q } \left( \sum_{(a,q)=1} \sum_{(b,r)=1} \frac{1}{K^2} \sum_{i,j =0}^{K-1} \int_0^1 \mathds{1}_{[- x_{q,i}, x_{q,i}]}(qy-a)\mathds{1}_{[-x_{r,j}, x_{r,j}]}(ry-b) \,\mathrm{d}y \right)^{k-1} 
 \\ = &\sum_{q,r \leq Q } \left( \frac{1}{K^2} \sum_{i,j =0}^{K-1} \int_0^1 \sum_{(a,q)=1}  \mathds{1}_{[-x_{q,i}, x_{q,i}]}(qy-a) \sum_{(b,r)=1} \mathds{1}_{[-x_{r,j}, x_{r,j}]}(ry-b) \,\mathrm{d}y \right)^{k-1} \\
  = &\sum_{q,r \leq Q } \left( \frac{1}{K^2} \sum_{i,j =0}^{K-1} \lambda_1( A_{q,i} \cap A_{r,j}) \right)^{k-1},
\end{align*}
where $A_{q,i}:= \bigcup_{(a,q)=1} \left[  \frac{a}{q}  - \frac{x_{q,i}}{q},  \frac{a}{q}  +\frac{x_{q,i}}{q} \right]$. In the last step, we used that the intervals $\left[  \frac{a}{q}  - \frac{x_{q,i}}{q},  \frac{a}{q}  +\frac{x_{q,i}}{q} \right]$ are disjoint for distinct $a$, since all $x_{q,i} \leq 2 g_k(q)^{1/2} \leq 1/2$, where the latter is ensured by \eqref{Eq_def_rho}.
Applying \cite[Lemma 5]{ABH2023}, we obtain the overlap estimate
\begin{equation*}
    \lambda_1(A_{q,i} \cap A_{r,j}) \leq \left( 1 + O \left( \log (D+2)^{-1} \right) \right)\lambda_1( A_{q,i}) \lambda_1(A_{r,j}) \prod_{ \substack{p | \frac{qr}{(q,r)^2} \\ p > A}} \left( 1 + \frac{1}{p-1} \right),
\end{equation*}
where
\begin{align}\notag
D & =D(q,r,i,j)= \frac{\max \{ r x_{q,i}, q x_{r,j} \}}{\gcd(q,r)},\\
\label{def_A}A & =A(D):= \exp\left(\frac{\log(D + 100)\log \log \log (D + 100)}{8 \log \log (D + 100)} + 1\right)
\end{align}
are is as in \cite[Equation $(10)$]{ABH2023}. We note that $A$ is monotonically increasing as a function in $D$, and $D(q,r,i,j) \geq \bar{D}(q,r)$ with 
\begin{align*}
\bar{D}(q,r) &:= \frac{\max\{ r \psi_{\varepsilon}^{(k)}(q), q \psi_{\varepsilon}^{(k)}(r)\}}{\gcd(q,r)}, \quad \psi_{\varepsilon}^{(k)}(q):= \left[g_k(q) \left(1 + \varepsilon \log \left( g_k(q)^{-1} \right)\right) \right]^{k-1}.
\end{align*}
Moreover, we introduce the quantity
\[
P_{\psi}(q,r) := \prod_{ \substack{p | \frac{qr}{(q,r)^2} \\ p > \bar{A}(q,r)}} \left( 1 + \frac{1}{p-1} \right),
\]
where $\bar{A}(q,r) := A(\bar{D})$.
This gives us the estimate
\[
\lambda_1(A_{q,i} \cap A_{r,j}) \leq \lambda_1( A_{q,i}) \lambda_1(A_{r,j})  \left( 1 + O \left( \log (\bar{D}(q,r)+2)^{-1} \right) \right)P_{\psi}(q,r).
\]
Furthermore, note that
\begin{align*}
\lim_{K \rightarrow \infty}  \left( \frac{1}{K^2} \sum_{i,j=0}^{K-1}\lambda_1( A_{q,i}) \lambda_1(A_{r,j}) \right)^{k-1} &= \int_{[0,1]^{k-1}} \gamma_{q,\varepsilon}^{(k)}(x) \,\mathrm{d}x \int_{[0,1]^{k-1}} \gamma_{r,\varepsilon}^{(k)}(x) \,\mathrm{d}x \\
& \leq (1 + \varepsilon)^{2(k-1)}  \int_{[0,1]^{k-1}} \gamma_q^{(k)}(x) \,\mathrm{d}x \int_{[0,1]^{k-1}} \gamma_r^{(k)}(x) \,\mathrm{d}x.
\end{align*}
Hence we get
\begin{align*}
    &\sum_{q,r \leq Q} \int_{[0,1]^{k-1}} \gamma_q^{(k)}(x) \gamma_r^{(k)}(x) \,\mathrm{d}x \\ \leq  &\sum_{q,r \leq Q} \left( \lim_{K \rightarrow \infty} \frac{1}{K^2} \sum_{i,j=0}^{K-1}\lambda_1( A_{q,i}) \lambda_1(A_{r,j}) \left( 1 + O \left( \log (\bar{D}(q,r)+2)^{-1} \right) \right) \prod_{ \substack{p | \frac{qr}{(q,r)^2} \\ p > \bar{A}(q,r)}} \left( 1 + \frac{1}{p-1} \right) \right)^{k-1} \\
    \leq &(1 + \varepsilon)^{2(k-1)} \sum_{q,r \leq Q} \left(\int_{[0,1]^{k-1}} \gamma_q^{(k)}(x) \,\mathrm{d}x\right)\left( \int_{[0,1]^{k-1}} \gamma_r^{(k)}(x) \,\mathrm{d}x\right)  \left[\left( 1 + O \left( \log (\bar{D}(q,r)+2)^{-1} \right) \right) P_{\psi}(q,r) \right]^{k-1}.
\end{align*}

Now let 
\[
\chi(q) := \frac{q}{\varphi(q)}\int_{[0,1]^{k-1}} \gamma_q^{(k)}(x)\,\mathrm{d}x = \left[ g_k(q)\left(1 + \frac{1}{2} \log \left(g_k(q)^{-1} \right )\right) \right]^{k-1} \leq \frac{1}{2},
\]
where the last inequality is true again by \eqref{Eq_def_rho}.
We observe that $\chi(q) \asymp_{\varepsilon} \psi_{\varepsilon}^{(k)}(q)$ and thus 
Proposition \ref{D_A_equiv} (stated below) gives us

\[\begin{split}
 &\sum_{q,r \leq Q} \int_0^1 \gamma_q^{(k)}(x) \gamma_r^{(k)}(x) \,\mathrm{d}x
  \\  & \leq (1 + \varepsilon)^{2(k-1)} \sum_{q,r \leq Q} \chi(q) \frac{\varphi(q)}{q} \chi(r) \frac{\varphi(r)}{r}  \left[ \left( 1 + O \left( \log (D_{\chi}(q,r)+2)^{-1} \right) \right) P_{\chi}(q,r) \right]^{k-1},
\end{split}
\]
where \[P_{\chi}(q,r)= \prod\limits_{ \substack{p | \frac{qr}{(q,r)^2} \\ p > A_{\chi}(q,r)}} \left( 1 + \frac{1}{p-1} \right)\]
and $A_{\chi}(q,r) = A(D_{\chi}(q,r))$ with 
$D_{\chi}(q,r) = \max\left\{\frac{r \chi(q)}{(q,r)},\frac{q\chi(r)}{(q,r)}\right\}$.
Writing 
\[
X(Q) := \sum_{q \leq Q} \chi(q)\frac{\varphi(q)}{q} ,
\]
we can now follow, up to insignificant changes, the proof of \cite[Theorem 2]{ABH2023} with the same decomposition into sets $\mathcal{E}_1 - \mathcal{E}_5$ for some $C > 11$ to establish  
\[
\sum_{q,r \leq Q} \int_{[0,1]^{k-1}} \gamma_q^{(k)}(x) \gamma_r^{(k)}(x) \,\mathrm{d}x \leq (1+ \varepsilon)^{2(k-1)} X(Q)^2 \left(1 + O_{\varepsilon}(\log (X(Q))^{-1})\right).
\]
Since 
\begin{align*}
X(Q) & = \sum \limits_{q \leq Q} \int_{[0,1]^{k-1}} \gamma_q^{(k)}(x)  \,\mathrm{d}x \\
&=\sum \limits_{q \leq Q} \left( \frac{\varphi(q)}{q} \right)^{k-1} \left(\int_{0}^1 f_q^{(k)}(y) \,\mathrm{d}y \right)^{k-1} \\
& \gg \sum \limits_{q \leq Q} \left( \frac{\varphi(q)}{q} \right)^{k-1} g_k(q)^{k-1} \left(\log g_k(q) \right)^{k-1} \\
& \gg  \sum \limits_{q \leq Q} \left( \frac{\varphi(q)}{q} \right)^{k} \psi(q) \left( \log \frac{q}{\varphi(q) \psi(q) } \right)^{k-1}, 
\end{align*}
we see that by assumption, $X(Q)$ tends to infinity as $Q \to \infty$. Therefore, we find $Q_{\varepsilon} \in \N$ large enough such that for all $Q \geq Q_{\varepsilon}$, we have

\[
\sum_{q,r \leq Q} \int_{[0,1]^{k-1}} \gamma_q^{(k)}(x) \gamma_r^{(k)}(x)  \,\mathrm{d}x \leq (1 + 2\varepsilon)^{2(k-1)}\left(\sum\limits_{q \leq Q} \int_{[0,1]^{k-1}} \gamma_q^{(k)}(x)  \,\mathrm{d}x\right)^2.
\]
Thus, the statement follows by choosing $\varepsilon$ small enough such that 
\[
(1 + 2\varepsilon)^{2(k-1)} \leq (1+ \delta).
\]
\end{proof}

\begin{proposition}\label{D_A_equiv}
    Let $\psi,\chi: \N \to [0,1/2]$ be two functions such that there exist constants $C_1,C_2 > 0$ with 
    \[C_1\chi(q) \leq \psi(q) \leq C_2\chi(q), \quad q \in \N.\]
    Further, denote by $D_{\psi}(q,r) = \max\left\{\frac{r\psi(q)}{(q,r)},\frac{q\psi(r)}{(q,r)}\right\}, 
     D_{\chi}(q,r) = \max\left\{\frac{r\chi(q)}{(q,r)},\frac{q\chi(r)}{(q,r)}\right\}$
    and $A_{\psi} = A(D_{\psi}),A_{\chi} = A(D_{\chi})$ with $A$ as in \eqref{def_A}.
    Then we have
    \begin{itemize}
        \item[] $(i)$ \quad  $\log(D_{\psi}+2) = \Theta_{C_1,C_2} \left( \log (D_{\chi}+2) \right)$,
        \item[] $(ii)$ \quad $\prod\limits_{\substack{p \mid \frac{qr}{(q,r)^2}\\ p > A_{\psi}}}\left(1 + \frac{1}{p}\right)
        = \prod\limits_{\substack{p \mid \frac{qr}{(q,r)^2}\\ p > A_{\chi}}}\left(1 + \frac{1}{p}\right)\left(1 + O_{C_1,C_2}\left(\frac{1}{\log (D_{\psi}+2)} \right)\right),
        $ 
        \item[] $(iii)$ \quad $\prod\limits_{\substack{p \mid \frac{qr}{(q,r)^2}\\ p > A_{\psi}}}\left(1 + \frac{1}{p-1}\right)
        = \prod\limits_{\substack{p \mid \frac{qr}{(q,r)^2}\\ p > A_{\psi}}}\left(1 + \frac{1}{p}\right)\left(1 + O\left(\frac{1}{\log (D_{\psi}+2)^{100}}\right)\right).
        $
    \end{itemize}
\end{proposition}
\begin{proof}
    The claim in $(i)$ is obvious, so we start with the proof of $(ii)$. Since $\psi$ and $\chi$ are (up to constants) equivalent, it suffices to show that
    \[
    \prod_{\substack{p | \frac{qr}{(q,r)^2} \\ A_{C_1\chi} \leq p \leq A_{C_2 \chi}}} \left( 1 + \frac{1}{p} \right) = 1 + O_{C_1, C_2}\left( \frac{1}{\log (D_{\chi}+2)} \right).
    \]
    Employing $(i)$, we get
    \[
    \prod_{\substack{p | \frac{qr}{(q,r)^2} \\ A_{C_1\chi} \leq p \leq A_{C_2 \chi}}}\left(1 + \frac{1}{p}\right)
    \leq \exp\left(\sum_{A_{C_1 \chi} \leq p \leq A_{C_2\chi}} \frac{1}{p}\right)
    \leq \exp\left(
    \log\left(\frac{\log A_{C_2\chi}}{\log A_{C_1 \chi}} + O\left(\frac{1}{\log(A_{\chi})^2}\right)\right)
    \right)
    \]
    where we have used Mertens' Theorem in the form of (for the error term see, e.g., \cite{Rosser_Schoenfeld_1962})
    \[\sum_{p \leq x} \frac{1}{p} = \log \log x + M + O\left(\exp\big(-a \sqrt{\log (x)}\big)\right),\]
    with $M$ being an absolute constant and $a > 0$.
    For any fixed constant $c > 0$, it follows from the definition of $A_{\chi}$ that
    \[
    \log A_{c \chi} = \log A_{\chi} \left( 1+O \left( \frac{1}{\log (D_{\chi }+2)} \right) \right), \qquad  \frac{1}{\log A_{\chi}^2}=O \left( \frac{1}{\log (D_{\chi}+2)}\right) 
    \]
    and hence,
    \[
    \exp\left(
    \log\left(\frac{\log A_{C_2\chi}}{\log A_{C_1 \chi}} + O\left(\frac{1}{\log(A_{\chi})^2}\right)\right)
    \right) = \exp\left(\log\left(1 + O\left(\frac{1}{\log (D_{\chi}+2)}\right)\right)\right) =1 + O\left(\frac{1}{\log (D_{\chi}+2)}\right).
    \]
This shows the claim in $(ii)$. For $(iii)$, we note that 
\[
\prod_{\substack{p \mid \frac{qr}{(q,r)^2}\\ p > A_{\psi}}}\left(1 + \frac{1}{p-1}\right)= \prod_{\substack{p \mid \frac{qr}{(q,r)^2}\\ p > A_{\psi}}}\left(1 + \frac{1}{p}\right) \prod_{\substack{p \mid \frac{qr}{(q,r)^2}\\ p > A_{\psi}}} \left( 1 + \frac{1}{p^2-1} \right).
\]
Since
\begin{align*}
 \prod_{\substack{p \mid \frac{qr}{(q,r)^2}\\ p > A_{\psi}}} \left( 1 + \frac{1}{p^2-1} \right) & = \exp \left(  O \left( \sum_{p > A_{\psi}} \frac{1}{p^2} \right) \right) \\
 & = \exp \left(  O \left( \frac{1}{A_{\psi}} \right) \right) \\
 & = 1 +O \left( \frac{1}{\log (D_{\psi}+2)^{100}} \right),
\end{align*}
with an absolute implied constant, the statement follows.
\end{proof}

\subsection{Proof of Corollaries \ref{littlewood_subsequence},\ref{Hausdorff_thm} and \ref{Hausdorff_cor}}

\begin{proof}[Proof of Corollary \ref{littlewood_subsequence}]
    Let $\varepsilon > 0$ fixed and let $\psi(q) = \psi_{A,\varepsilon}(q) := \mathds{1}_{[q \in A]}\frac{\varepsilon}{q}$. Assume first that
$\frac{\#\{n \leq N: n \in A\}}{N} \gg \frac{\log \log N}{(\log N)^2}$.    
By Corollary \ref{weak_thm}, it suffices to show that
\[\sum_{q \in \N} \mathds{1}_{[q \in A]}\frac{1}{q (\log \log q)^2}\log(q(\log\log q)^2) 
\gg \sum_{q \in \N} \mathds{1}_{[q \in A]}\frac{\log(q)}{q (\log \log q)^2} =
\infty.\]
By summation by parts and the mean value theorem this holds true since

\[\sum_{q \in N} \frac{1}{q\log q \log \log q} = \infty.\]

If $\frac{\#\{n \leq N: n \in A\}}{N} \gg \frac{\log \log N}{(\log N)^2}$, we use the fact that the convergence part of Conjecture \ref{conj_normal} holds true (for a proof see \cite{Beresnevich_Haynes_Velani2013}), and thus we only need to show that 
\[\sum_{q \in \N} \mathds{1}_{[q \in A]}\frac{1}{q}\log(q) < \infty,\]
which follows again after applying summation by parts.
\end{proof}

\begin{proof}[Proof of Corollaries \ref{Hausdorff_thm} and \ref{Hausdorff_cor}]
It was already shown by Hussain and Simmons in \cite{Mumtaz_Simmons2018} that even for non-monotonic $\psi$
\begin{equation}
    \label{upper_bound_HD}
\dim_H(S_k^{\times}) \leq (k-1) + d(\psi),
\end{equation}
and since clearly $D_k^{\times}(\psi) \subseteq S_k^{\times}(\psi)$, the proof of Corollary \ref{Hausdorff_thm} is reduced to showing
$\dim_H(D_k^{\times}) \geq (k-1) + d(\psi)$. To prove this, we make use of the Hausdorff version of the Duffin-Schaeffer conjecture due to Beresnevich-Velani \cite{beresnevich_velani_2006}, who proved that
\begin{equation}\label{HD_DS}\dim_H(D_1^{\times}(\psi)) = \min\{d(\psi),1\}.
\end{equation}
In order to compute the Hausdorff-dimension of $D_k^{\times}(\psi)$ for arbitrary $k \geq 1$, we need the following claim: For any $\varepsilon > 0$, there exists a constant $c= c_{\varepsilon} > 0$ such that for all $q \in \N$
\begin{equation}
    \label{max_distance_coprime}
|| q \alpha _i ||' \leq c_{\varepsilon} q^{\varepsilon}.
\end{equation}
 This follows from
\[
\begin{split}\sum_{\substack{n \leq x \\ (n,q) = 1}}1 &= \sum_{n \leq x} \sum_{d \mid (n,q)}\mu(d)
    = \sum_{d \mid q} \mu(d)\sum_{\substack{n \leq x \\ d \mid n }} 1 = \sum_{d \mid q} \mu(d) \left(\frac{x}{d} + O(1)\right)
    \\&= x \sum_{d \mid q} \frac{\mu(d)}{d} + O(\tau(q)) = x\frac{\varphi(q)}{q} + O(\tau(q)),
\end{split}
\]

which implies that for $(y-x) \gg_{\varepsilon} q^{\varepsilon}$, we have
\[
\sum_{\substack{x \leq n \leq y\\ (n,q) = 1}} 1 > 0,
\]   

thus \eqref{max_distance_coprime} follows. Hence we have
\begin{align*}
D_k^{\times}(\psi) &= \left\{ \alpha \in [0,1]^k  : \prod_{i=1}^k || q \alpha_i ||' \leq \psi(q) \right\} \\
& \supseteq\left\{
\alpha \in [0,1]^k  : || q \alpha_1 ||' \leq \frac{\psi(q)}{c^{k-1}q^{(k-1) \varepsilon}} \right\} \\
& = [0,1]^{k-1} \times D_1^{\times}(\psi/f_{\varepsilon}),
\end{align*}
where $f_{\varepsilon}(q) := c^{k-1}q^{(k-1)\varepsilon}$. By elementary properties of the Hausdorff dimension, the definition of $d(\psi)$ and  \eqref{HD_DS}, we get that for all $\varepsilon > 0$,
\[\dim_H(D_k^{\times}(\psi)) \geq (k-1) + \dim_H\left(D_1^{\times}\left(\frac{\psi}{f_{\varepsilon}}\right)\right)
\geq  (k-1) + \dim_H(D_1^{\times}\left(\psi\right)) - \varepsilon,
\]
where we used \eqref{HD_DS} and the fact that $\psi$ is absolutely bounded.
Therefore, it follows that
$\dim_H(D_k^{\times}(\psi)) \geq (k-1) + d(\psi)$, which proves Corollary \ref{Hausdorff_thm}. Corollary 
\ref{Hausdorff_cor} can now be deduced from $q/\varphi(q) \ll_{\varepsilon} q^{\varepsilon}$ and an analogous argumentation.
\end{proof}

\subsection*{Acknowledgements} 
The authors would like to thank Christoph Aistleitner, Bence Borda, Victor Beresnevich, Niclas Technau, and Sanju Velani for various valuable discussions and comments on an earlier version of this manuscript.
LF was supported by the Austrian Science Fund (FWF) Project P 35322. MH was supported by the EPSRC grant EP/X030784/1.
A part of this work was supported by the Swedish Research Council under grant no. 2016-06596 while MH was in residence at Institut Mittag-Leffler in Djursholm, Sweden in 2024.

\bibliographystyle{plain}
\bibliography{bibliographie.bib}

\begin{thebibliography}{10}

\bibitem{Aistleitner_2014}
C.~Aistleitner.
\newblock A note on the {D}uffin-{S}chaeffer conjecture with slow divergence.
\newblock {\em Bull. Lond. Math. Soc}, 46(1):164–168, 2014.

\bibitem{ABH2023}
C.~Aistleitner, B.~Borda, and M.~Hauke.
\newblock On the metric theory of approximations by reduced fractions: a
  quantitative {K}oukoulopoulos-{M}aynard theorem.
\newblock {\em Compos. Math.}, 159(2):207--231, 2023.

\bibitem{Aistleitner_Lachmann_Munsch_Technau_Zafeiropoulos_2019}
C.~Aistleitner, T.~Lachmann, M.~Munsch, N.~Technau, and A.~Zafeiropoulos.
\newblock The {D}uffin-{S}chaeffer conjecture with extra divergence.
\newblock {\em Adv. Math.}, 356(106808), 2019.

\bibitem{Allen_Ramirez}
D.~Allen and F.~Ram\'{\i}rez.
\newblock Independence inheritance and {D}iophantine approximation for systems
  of linear forms.
\newblock {\em Int. Math. Res. Not. IMRN}, (2):1760--1794, 2023.

\bibitem{BadVel2011}
D.~Badziahin and S.~Velani.
\newblock Multiplicatively badly approximable numbers and generalised cantor
  sets.
\newblock {\em Adv. Math.}, 228(5):2766--2796, 2011.

\bibitem{Beresnevich_Harman_Haynes_Velani_2013}
V.~Beresnevich, G.~Harman, A.~Haynes, and S.~Velani.
\newblock The {D}uffin–{S}chaeffer conjecture with extra divergence {II}.
\newblock {\em Math.Z.}, 275(1–2):127–133, 2013.

\bibitem{Beresnevich_Haynes_Velani2013}
V.~Beresnevich, A.~Haynes, and S.~Velani.
\newblock Multiplicative zero-one laws and metric number theory.
\newblock {\em Acta Arith.}, 160(2):101--114, 2013.

\bibitem{beresnevich_velani_2006}
V.~Beresnevich and S.~Velani.
\newblock A mass transference principle and the {D}uffin-{S}chaeffer conjecture
  for {H}ausdorff measures.
\newblock {\em Ann. of Math. (2)}, 164(3):971--992, 2006.

\bibitem{Bovey_Dodson_1978}
J.~Bovey and M.~Dodson.
\newblock The fractional dimension of sets whose simultaneous rational
  approximations have errors with a small product.
\newblock {\em Bull. Lond. Math. Soc}, 10(2):213–218, 1978.

\bibitem{Chow_2018}
S.~Chow.
\newblock Bohr sets and multiplicative {D}iophantine approximation.
\newblock {\em Duke Math. J.}, 167(9):1623--1642, 2018.

\bibitem{Chow_Technau_2020}
S.~Chow and N.~Technau.
\newblock Littlewood and {D}uffin--{S}chaeffer-type problems in diophantine
  approximation.
\newblock {\em preprint: arXiv:2010.09069, to appear in: Mem. Amer. Math. Soc},
  2020.

\bibitem{Chow_Technau_2023}
S.~Chow and N.~Technau.
\newblock Dispersion and {L}ittlewood’s conjecture.
\newblock {\em preprint: arXiv:2307.14871}, 2023.

\bibitem{Chow_Zafeiropoulos_2021}
S.~Chow and A.~Zafeiropoulos.
\newblock Fully inhomogeneous multiplicative {D}iophantine approximation of
  badly approximable numbers.
\newblock {\em Mathematika}, 67(3):639–646, 2021.

\bibitem{Dodson_1991}
M.~Dodson.
\newblock Star bodies and {D}iophantine approximation.
\newblock {\em J. London Math. Soc. (2)}, 44(1):1--8, 1991.

\bibitem{Dodson_Kristensen_2006}
M.~Dodson and S.~Kristensen.
\newblock Khintchine’s theorem and transference principle for star bodies.
\newblock {\em Int. J. Number Theory}, 02(03):431–453, 2006.

\bibitem{Duffin_Schaeffer_1941}
R.~Duffin and A.~Schaeffer.
\newblock Khintchine’s problem in metric {D}iophantine approximation.
\newblock {\em Duke Math. J.}, 8(2):243–255, 1941.

\bibitem{Einsiedler2006}
M.~Einsiedler, A.~Katok, and E.~Lindenstrauss.
\newblock Invariant measures and the set of exceptions to {L}ittlewood’s
  conjecture.
\newblock {\em Ann. of Math. (2)}, 164:513–560, 2006.

\bibitem{Erdos_1970}
P.~Erd\H{o}s.
\newblock On the distribution of the convergents of almost all real numbers.
\newblock {\em J. Number Theory}, 2(4):425–441, 1970.

\bibitem{Gallagher_1961}
P.~Gallagher.
\newblock Approximation by reduced fractions.
\newblock {\em J. Math. Soc. Jpn}, 13(4):342–345, 1961.

\bibitem{Gallagher_1962}
P.~Gallagher.
\newblock Metric simultaneous {D}iophantine approximation.
\newblock {\em J. Lond. Math. Soc. (2)}, s1-37(1):387–390, 1962.

\bibitem{Harman_1998}
G.~Harman.
\newblock {\em Metric number theory}.
\newblock Clarendon Press, Oxford, England, 1998.

\bibitem{Hauke}
M.~Hauke.
\newblock Quantitative inhomogeneous {D}iophantine approximation for systems of
  linear forms.
\newblock {\em preprint: arXiv:2312.01986}, 2023.

\bibitem{Haynes_Pollington_Velani_2012}
A.~Haynes, A.~Pollington, and S.~Velani.
\newblock The {D}uffin–{S}chaeffer {C}onjecture with extra divergence.
\newblock {\em Math. Ann.}, 353(2):259–273, 2012.

\bibitem{Mumtaz_Simmons2018}
M.~Hussain and D.~Simmons.
\newblock The {H}ausdorff measure version of {G}allagher's theorem---closing
  the gap and beyond.
\newblock {\em J. Number Theory}, 186:211--225, 2018.

\bibitem{Koukoulopoulos_Maynard_2020}
D.~Koukoulopoulos and J.~Maynard.
\newblock On the {D}uffin-{S}chaeffer conjecture.
\newblock {\em Ann. of Math. (2)}, 192(1):251, 2020.

\bibitem{Kristensen_2004}
S.~Kristensen.
\newblock Metric {D}iophantine approximation with respect to planar distance
  functions.
\newblock {\em preprint: arXiv:0401371}, 2004.

\bibitem{PollingtonVaughan1990}
A.~Pollington and R.~Vaughan.
\newblock The k-dimensional {D}uffin and {S}chaeffer conjecture.
\newblock {\em Mathematika}, 37(2):190–200, 1990.

\bibitem{Pollington_Velani_2000}
A.~Pollington and S.~Velani.
\newblock On a problem in simultaneous {D}iophantine approximation:
  {L}ittlewood’s conjecture.
\newblock {\em Acta Math.}, 185(2):287–306, 2000.

\bibitem{PVZZ2022}
A.~D. Pollington, S.~Velani, A.~Zafeiropoulos, and E.~Zorin.
\newblock Inhomogeneous diophantine approximation on $m_0$-sets with restricted
  denominators.
\newblock {\em Int. Math. Res. Not.}, 2022(11):8571--8643, 2022.

\bibitem{Ramirez_counter}
F.~Ram\'{\i}rez.
\newblock Counterexamples, covering systems, and zero-one laws for
  inhomogeneous approximation.
\newblock {\em Int. J. Number Theory}, 13(3):633--654, 2017.

\bibitem{Rosser_Schoenfeld_1962}
J.~Rosser and L.~Schoenfeld.
\newblock Approximate formulas for some functions of prime numbers.
\newblock {\em Ill. J. Math.}, 6(1):64–94, 1962.

\bibitem{Szusz_1958}
P.~Sz\"{u}sz.
\newblock \"{U}ber die metrische {T}heorie der {D}iophantischen
  {A}pproximation.
\newblock {\em Acta Math. Acad. Sci. Hungar.}, 9:177--193, 1958.

\bibitem{Vaaler_1978}
J.~Vaaler.
\newblock On the metric theory of {D}iophantine approximation.
\newblock {\em Pac. J. Math.}, 76(2):527–539, 1978.

\bibitem{Wang_Yu1981}
Y.~Wang and K.~R. Yu.
\newblock A note on some metrical theorems in {D}iophantine approximation.
\newblock {\em Chinese Ann. Math.}, 2(1):1--12, 1981.

\bibitem{Yu1}
H.~Yu.
\newblock A {F}ourier-analytic approach to inhomogeneous {D}iophantine
  approximation.
\newblock {\em Acta Arith.}, 190(3):263--292, 2019.

\bibitem{Yu2}
H.~Yu.
\newblock On the metric theory of inhomogeneous {D}iophantine approximation: an
  {E}rd{\H{o}}s-{V}aaler type result.
\newblock {\em J. Number Theory}, 224:243--273, 2021.

\bibitem{Yu3}
H.~Yu.
\newblock On the metric theory of multiplicative {D}iophantine approximation.
\newblock {\em J. Anal. Math.}, 149(2):763--800, 2023.

\end{thebibliography}
\vspace{1cm}

\author{Lorenz Fr\"uhwirth}
{\footnotesize

Graz University of Technology

Steyrergasse 30, 8010 Graz, Austria

Email: \texttt{fruehwirth@math.tugraz.at}

}
\vspace{5mm}

\author{Manuel Hauke}
{\footnotesize 

University of York

Department of Mathematics

YO10 5DD York, United Kingdom

Email: \texttt{manuel.hauke@york.ac.uk}}

\end{document}